\documentclass[a4paper,11pt]{amsart}
\usepackage[T1]{fontenc}
\usepackage{mathtools,amssymb,needspace}
\usepackage[hidelinks]{hyperref}

\newtheorem{theorem}{Theorem}[section]
\newtheorem{proposition}[theorem]{Proposition}
\newtheorem{lemma}[theorem]{Lemma}
\newtheorem{corollary}[theorem]{Corollary}
\newtheorem{question}[theorem]{Question}
\theoremstyle{definition}
\newtheorem{definition}[theorem]{Definition}
\newtheorem{example}[theorem]{Example}
\theoremstyle{remark}
\newtheorem{remark}[theorem]{Remark}

\newcommand{\Pow}{\mathcal P}
\newcommand{\dom}{\mathrel{\triangleright}}
\newcommand{\D}{D}
\newcommand{\branch}{\mathbin{\boxtimes}}
\newcommand{\set}[1]{\left\{#1\right\}}
\renewcommand\iff{\quad\text{iff}\quad}

\title{Generalized Quantifiers: Scope Dominance and Branching}
\author{Fredrik Engstr\"om}
\address{Department of Philosophy, Linguistics and Theory of Science,
University of Gothenburg, Box 200, SE-405 30 G\"oteborg, Sweden}
\email{fredrik.engstrom@gu.se}
\date{\today}

\begin{document}

\begin{abstract}
Let $P$ and $Q$ be proper upward monotone unary generalized quantifiers on a
fixed domain $D$.  We study the validity, over all structures $M$ with domain $D$, of 
\[
M \models Px\,Qy\,R(x,y) \rightarrow  Qy\,Px\,R(x,y),
\] 
called \emph{scope dominance}. The iteration of two quantifiers lies between
their branching product, which is generated by rectangles, and the dual of
the branching product of their duals. We give criteria for equality with
these bounds, obtain a new proof of the countable characterization of scope dominance, and
characterize the case in which the inner quantifier is ``at least $\kappa$
many.'' Under the Generalized Continuum Hypothesis, Goldberg's theorem shows
that, for ultrafilters, scope dominance is equivalent to equality between
iteration and the branching product, whereas an elementary counterexample on
a domain of size $\aleph_2$ shows that this equivalence fails for arbitrary
upward monotone quantifiers. For the case where both quantifiers are filters,
the question becomes set-theoretic: the existence of a counterexample is
equiconsistent with the existence of a measurable cardinal.
\end{abstract}

\maketitle

\section{Introduction}

A unary generalized quantifier on a domain $\D$ is a family of subsets of
$\D$; see \cite{Mostowski1957,Lindstrom1966}.  
We work with local quantifiers on a fixed domain and do not require
invariance under permutations of that domain.
Iterating two such
quantifiers gives the familiar two scope readings of a transitive
construction \cite{PetersWesterstahl2006}.  Following Westerst\aa hl
\cite{Westerstahl1986} and Altman, Peterzil, and Winter
\cite{AltmanPeterzilWinter2005}, we say that $P$ \emph{dominates} $Q$ when
the $PQ$ reading entails the $QP$ reading for every binary relation.

The elementary example is that $\exists$ dominates $\forall$ on every
nonempty domain:
$\models_{\D}\exists x\,\forall y\,R(x,y)\rightarrow
\forall y\,\exists x\,R(x,y)$.
The converse fails whenever $|\D|\geq 2$:
$\not\models_{\D}\forall y\,\exists x\,R(x,y)\rightarrow
\exists x\,\forall y\,R(x,y)$.
Here $\models_{\D}$ denotes validity over all structures with domain $\D$.

Alongside the two linear scope readings there is a branching reading
\cite{PetersWesterstahl2006}.  For upward monotone quantifiers $P$ and $Q$,
this reading is true of $\varphi(x,y)$ in a structure $M$ with domain $\D$
exactly when there are $A\in P$ and $B\in Q$ such that
$M\models\varphi(a,b)$ for every $(a,b)\in A\times B$.  This defines the
\emph{branching product} $P\branch Q$.  Unlike iteration, the branching
product selects the two sets independently.  It is symmetric under
transposition: a relation $R$ belongs to $P\branch Q$ if and only if
$R^{-1}$ belongs to $Q\branch P$.  The central question of the paper is
whether scope dominance forces the iteration of either $(P,Q)$ or the
reversed pair of duals $(Q^d,P^d)$ to coincide with its branching product.

The finite-domain theory of scope dominance goes back to Westerst\aa
hl \cite{Westerstahl1986}, who treated upward monotone global determiners
satisfying conservativity, permutation invariance, and extension. For proper
upward monotone local quantifiers, the finite-domain existential--universal
characterization is stated and proved in \cite{AltmanPeterzilWinter2005}; see also \cite{BenAviWinter2009}. Ben-Avi and Winter \cite{BenAviWinter2004}
characterized dominance for mixed upward/downward pairs on finite domains.
They also gave numerical characterizations for pairs of monotone quantifiers
obtained by fixing the first argument of a conservative,
permutation-invariant determiner. Altman, Peterzil, and Winter extended the
upward-monotone characterization to countable domains
\cite{AltmanPeterzilWinter2005}, showing that infinite domains admit dominance
 relations beyond the finite existential--universal alternatives. Westerst\aa
 hl \cite{Westerstahl1996} characterized self-commuting quantifiers on
 arbitrary domains, without assuming monotonicity. In their 2009 survey,
 Ben-Avi and Winter \cite{BenAviWinter2009} identified full characterizations
 of scope dominance for arbitrary quantifiers and over domains of arbitrary
 cardinality as open problems.

The theory of generalized quantifiers originates with Mostowski
\cite{Mostowski1957} and Lindstr\"om \cite{Lindstrom1966}. For a comprehensive
logical and linguistic background, see Peters and Westerst\aa hl
\cite{PetersWesterstahl2006}.

Goldberg's recent work on products of ultrafilters provides a direct connection
with the branching-product question. For ultrafilters, iteration is the
tensor (or Fubini) product, the branching product is the Cartesian
product filter, and scope dominance is equivalent to commutativity of
the tensor product under coordinate transposition. Goldberg proved
under GCH that two ultrafilters commute precisely when their Cartesian
product is an ultrafilter, equivalently when their tensor and Cartesian
products coincide \cite[Theorem~2.16]{Goldberg2025}. Thus the ultrafilter
instance of the branching-product dichotomy holds under GCH. Since
ultrafilters are self-dual, the two branching alternatives are equivalent
in this case. 

There are also positive cases outside the ultrafilter setting. If the
inner quantifier means ``at least $\kappa$ many'', where
$1\leq\kappa\leq|\D|$ and $\D$ is infinite, dominance forces the lower
branching equality; see Proposition~\ref{prop:inner-cardinality}.
Nevertheless, the counterexample in
Section~\ref{sec:counterexample} shows in ZFC that dominance can hold
while both branching equalities fail, even when the inner quantifier is
the co-countable filter. Its outer quantifier is not a filter. Thus the
ultrafilter theorem does not extend to arbitrary upward monotone
families, even when the inner quantifier remains a filter.

The restriction to two filters has a sharp set-theoretic answer.  Under
$V=L$, and more generally when there is no inner model with a measurable
cardinal, dominance between filters forces the lower branching equality.
We derive this from Donder's regularity theorem \cite{Donder1988}, in the
form recorded in \cite[Theorem~8.1]{Usuba2026}, together with the chain
alternatives developed below.  Conversely, a measurable cardinal yields
two non-maximal filters for which dominance holds but both branching
equalities fail.  The existence of a two-filter counterexample is therefore
equiconsistent with a measurable cardinal.  Unlike the unrestricted case,
the two-filter case is independent of ZFC, relative to that consistency
assumption.

The countable characterization is due to Altman, Peterzil, and
Winter \cite{AltmanPeterzilWinter2005}. 
In Theorem 5 of that paper, the mixed sufficiency argument invokes
finite-intersection closure of $Q_1$, although the hypothesis concerns $Q_1^d$. 
We give a corrected proof that exposes the two branching
alternatives and extends its main combinatorial ingredients to arbitrary
cardinalities. These extensions also explain where the countable argument
stops: at each regular cardinal dominance supplies a disjunction, whereas
the smallness argument needs a uniform choice across all relevant
cardinals. The counterexample makes a related obstruction
explicit: unions of outer quantifiers preserve dominance, even when their
components use different branching alternatives.

Section~\ref{sec:iteration} introduces iteration, duality, and the branching
product. Section~\ref{sec:goldberg} treats ultrafilters and Goldberg's
theorem.  Section~\ref{sec:completeness} develops the completeness machinery and the positive
inner-cardinality case.  The necessary conditions in Section~\ref{sec:dominance} 
are used in Section~\ref{sec:countable} to recover the
countable characterization.  Section~\ref{sec:counterexample} gives the ZFC
counterexample and explains how mixing the two branching mechanisms defeats
the unrestricted dichotomy. Section~\ref{sec:filters} establishes the
positive theorem under $V=L$, the measurable-cardinal counterexample for two
filters, and the resulting relative independence and equiconsistency.

We work in ZFC throughout, with additional hypotheses indicated explicitly.
All quantifiers below are local (they live on one fixed nonempty domain) and
upward monotone.  Unless explicitly stated otherwise, they are also
\emph{proper}: $\varnothing\notin Q$ and $\D\in Q$.  

Parts of this paper were developed several years ago. More recently, the author
used ChatGPT (OpenAI) to revisit and extend these investigations, including
developing the counterexamples and the two-filter results. A more detailed account
of this assistance appears in the AI declaration.

\section{Iteration, duality, and branching products}\label{sec:iteration}

For a relation $R\subseteq\D^2$, put
\[
 R_x=\set{y\in\D:(x,y)\in R},\qquad
 R^y=\set{x\in\D:(x,y)\in R},
\]
and let $R^{-1}=\set{(y,x):(x,y)\in R}$.

\begin{definition}
For quantifiers $P,Q\subseteq\Pow(\D)$, their \emph{iteration} is
\[
 P\cdot Q
 =\set{R\subseteq\D^2:\set{x:R_x\in Q}\in P}.
\]
We say that $P$ \emph{dominates} $Q$, and write $P\dom Q$, if
\[
 R\in P\cdot Q\quad\Longrightarrow\quad
 R^{-1}\in Q\cdot P
\]
for every $R\subseteq\D^2$.
\end{definition}

The preceding semantic reading has the following set-theoretic form.  The
branching product is the least upward monotone binary quantifier containing
all Cartesian products of a $P$-large set and a $Q$-large set.

\begin{definition}
Define
\[
 P\branch Q
 =\set{R\subseteq\D^2:(\exists A\in P)(\exists B\in Q)\,
 A\times B\subseteq R}.
\]
The dual of $Q$ is
\[
 Q^d=\set{A\subseteq\D:\D\setminus A\notin Q}.
\]
For a binary quantifier $\mathcal Q\subseteq\Pow(\D^2)$, its dual is
defined using complementation in $\D^2$.
\end{definition}

The following criterion will be used repeatedly. 

\begin{proposition}\label{prop:branching}
The following are equivalent.
\begin{enumerate}
\item $P\cdot Q=P\branch Q$.
\item Whenever $A\in P$ and $A_x\in Q$ for every $x\in A$, there is
      $C\in P$ with $C\subseteq A$ such that
      $\bigcap_{x\in C}A_x\in Q$.
\end{enumerate}
\end{proposition}
\begin{proof}
Assume the equality and let
$R=\bigcup_{x\in A}(\{x\}\times A_x)$.  Then
$R\in P\cdot Q$, so some $C\in P$ and $B\in Q$ satisfy
$C\times B\subseteq R$.  Properness of $Q$ implies $C\subseteq A$, and
$B\subseteq\bigcap_{x\in C}A_x$.  Upward monotonicity gives the desired
intersection.

Conversely, if $R\in P\cdot Q$, let
$A=\set{x:R_x\in Q}\in P$ and apply (2) to $(R_x)_{x\in A}$.
If $C$ is supplied by (2) and $B=\bigcap_{x\in C}R_x$, then
$C\times B\subseteq R$.
\end{proof}

\begin{proposition}\label{prop:duality}
For upward monotone quantifiers $P,Q$,
\begin{align*}
 (P\cdot Q)^d&=P^d\cdot Q^d,\\
 P\branch Q&\subseteq P\cdot Q
 \subseteq (P^d\branch Q^d)^d.
\end{align*}
Moreover,
\[
 P\dom Q\iff Q^d\dom P^d.
\]
\end{proposition}
\begin{proof}
Writing $R^c=\D^2\setminus R$, we have
\begin{align*}
R\in(P\cdot Q)^d
&\iff
 \set{x:(R^c)_x\in Q}\notin P\\
&\iff 
 \set{x:R_x\in Q^d}\in P^d.
\end{align*}
This proves the identity.  The left inclusion is immediate; applying it to
the duals and taking duals gives the right inclusion.  Finally, if
$P\dom Q$ and $R\in Q^d\cdot P^d=(Q\cdot P)^d$, then
$R^c\notin Q\cdot P$.  Hence $(R^{-1})^c\notin P\cdot Q$, so
$R^{-1}\in P^d\cdot Q^d$.  Apply the same argument to the duals for the
converse.
\end{proof}

Thus iteration has a least possible value, generated by rectangles, and a
greatest possible value, obtained by dualizing the branching product of
the dual quantifiers. Both bounds are attainable: $\exists\cdot\forall$
attains the lower bound, while $\forall\cdot\exists$ attains the upper bound.

\begin{question}\label{con:main}
If $P\dom Q$, must at least one of the following hold?
\begin{align}
 P\cdot Q&=P\branch Q,\label{eq:lower}\tag{L}\\
 Q^d\cdot P^d&=Q^d\branch P^d.\label{eq:upper}\tag{U}
\end{align}
Equivalently, in \eqref{eq:upper} one may write
$Q\cdot P=(Q^d\branch P^d)^d$.
\end{question}

Either equality is sufficient for dominance. To see this, suppose first
that \eqref{eq:lower} holds, and let $R\in P\cdot Q$. By the equality,
there are $A\in P$ and $B\in Q$ such that $A\times B\subseteq R$.
Transposing gives $B\times A\subseteq R^{-1}$. In particular, for every
$y\in B$ we have $A\subseteq R^y$, so upward monotonicity of $P$ gives
$R^y\in P$. Consequently,
\[
 B\subseteq\{y\in\D:R^y\in P\}.
\]
Since $B\in Q$, upward monotonicity of $Q$ implies that the set on the
right belongs to $Q$. Thus $R^{-1}\in Q\cdot P$, proving $P\dom Q$.

If \eqref{eq:upper} holds, applying the preceding argument to $Q^d$ and $P^d$ gives $Q^d \dom P^d$, and hence $P \dom Q$ by Proposition \ref{prop:duality}.

The issue is therefore necessity: must every instance of dominance arise
from at least one of these two branching equalities? The positive cases
below motivate the question, while Theorem~\ref{thm:counterexample} shows
that, in full generality, dominance can hold while both equalities fail.

\section{Branching products of ultrafilters}\label{sec:goldberg}

The preceding terminology has a standard filter-theoretic counterpart.  Here
a filter is proper, upward monotone, and closed under finite intersections.  If
$U$ and $W$ are filters, their Cartesian product $U\times W$ is the filter
generated by the rectangles $A\times B$, where $A\in U$ and $B\in W$.
Since a finite intersection of such rectangles is again such a rectangle,
as families of subsets of $\D^2$ we have
\[
 U\times W=U\branch W.
\]
Their left and right tensor products, in Goldberg's notation, are
\begin{align*}
 U\ltimes W&=\{R:\{x:R_x\in W\}\in U\}=U\cdot W,\\
 U\rtimes W&=\{R:\{y:R^y\in U\}\in W\}.
\end{align*}
Thus $U\rtimes W$ is the transpose of $W\cdot U$.

Following Goldberg \cite{Goldberg2025}, say that $W$ is \emph{$U$-closed} if
for every family $(B_x)_{x\in\D}$ of members of $W$, there is
$A\in U$ such that $\bigcap_{x\in A}B_x\in W$.
Proposition~\ref{prop:branching}, specialized
to filters, is precisely the equivalence
\[
 W\text{ is }U\text{-closed}
\iff 
 U\ltimes W=U\times W.
\]
This is also \cite[Lemma~2.4]{Goldberg2025}.

For ultrafilters, both tensor products are ultrafilters on $\D^2$.  Hence
\[
 U\dom W
 \iff 
 U\ltimes W\subseteq U\rtimes W
 \iff
 U\ltimes W=U\rtimes W.
\]

Consequently, Question~\ref{con:main} asks, in the ultrafilter case,
whether tensor commutativity forces equality with the Cartesian
product. Under dominance, self-duality and transposition make
the two alternatives \eqref{eq:lower} and \eqref{eq:upper} equivalent.

To isolate the role of GCH, let $\lambda\leq\eta$ be infinite cardinals.
An ultrafilter $U$ is \emph{$(\lambda,\eta)$-indecomposable} if every
family $\mathcal A$ with $|\mathcal A|\leq\eta$ and
$\bigcup\mathcal A\in U$ has a subfamily of cardinality less than
$\lambda$ whose union belongs to $U$.

\begin{theorem}[Goldberg \cite{Goldberg2025}]\label{thm:goldberg}
Let $U,W$ be ultrafilters.
\begin{enumerate}
\item If $U\ltimes W=U\rtimes W$, then for every infinite cardinal $\lambda$, one
of $U,W$ is $(\lambda,\lambda^+)$-indecomposable.
\item If for every infinite cardinal $\lambda$, one of $U,W$ is
$(\lambda,2^\lambda)$-indecomposable, then $U\times W$ is an ultrafilter.
\end{enumerate}
Consequently, under GCH, tensor commutativity is equivalent to
$U\ltimes W=U\times W=U\rtimes W$.
\end{theorem}

\begin{proof}
The two assertions are Theorems~2.15 and 2.14 of
\cite{Goldberg2025}.  Under GCH, $2^\lambda=\lambda^+$, so the necessary
condition in (1) supplies the sufficient condition in (2).  Since
$U\times W$ is contained in both tensor products, its being an ultrafilter
forces equality with both.
\end{proof}

\begin{corollary}[GCH]
\label{cor:goldberg-quantifiers}
Assume GCH and let $U,W$ be ultrafilters on $\D$, regarded as upward
monotone generalized quantifiers.  Then
\[
 U\dom W\iff U\cdot W=U\branch W.
\]
When these conditions hold, also $W\cdot U=W\branch U$.
\end{corollary}

\begin{proof}
The equivalence follows from Theorem~\ref{thm:goldberg} and the
identifications above: equality $U\cdot W=U\branch W$ makes
$U\times W$ an ultrafilter, hence equal to both tensor products.
Transposing $U\rtimes W=U\times W$ gives the final assertion.
\end{proof}

Thus GCH gives a positive answer to Question~\ref{con:main} for
ultrafilters. For general upward monotone quantifiers, the iterates
need not be ultrafilters, so the maximality argument above is
unavailable.

\section{Completeness and smallness}\label{sec:completeness}

We separate ordinary completeness from chain completeness.  This also avoids
the ambiguity that ``$\kappa$-complete'' sometimes means closure under
$\kappa$ intersections and sometimes closure under fewer than $\kappa$
intersections.

\begin{definition}
Let $\kappa$ be an infinite cardinal.
\begin{enumerate}
\item $Q$ is a \emph{filter} if it is proper, upward monotone, and closed
      under finite intersections.
\item $Q$ is \emph{$\kappa$-complete} if every intersection of fewer than
      $\kappa$ members of $Q$ belongs to $Q$.
\item $Q$ is \emph{$\kappa$-chain-complete} if
      $\bigcap_{\alpha<\kappa}X_\alpha\in Q$ for every descending chain
      $(X_\alpha)_{\alpha<\kappa}$ in $Q$.
\item $Q$ is \emph{$\kappa$-small} if every $A\in Q$ has a subset
      $B\in Q$ with $|B|<\kappa$.
\end{enumerate}
\end{definition}

Chain completeness depends only on cofinality: replacing an ordinal by a
cofinal subsequence does not change the intersection.  It is therefore enough
to record it at regular cardinals.

\begin{lemma}
\label{lem:chain-to-complete}
Let $F$ be a filter.  If $F$ is $\lambda$-chain-complete for every infinite
regular $\lambda<\kappa$, then $F$ is $\kappa$-complete.
\end{lemma}

\begin{proof}
Let $(X_\alpha)_{\alpha<\mu}$ be a family in $F$, where $\mu<\kappa$.
Inductively put $Y_\alpha=\bigcap_{\beta<\alpha}X_\beta$.  Successor steps
use finite-intersection closure.  At a limit stage, pass to a cofinal
subsequence and use chain completeness at the regular cardinal
$\operatorname{cf}(\alpha)$.  The same argument at the end gives
$\bigcap_{\alpha<\mu}X_\alpha\in F$.
\end{proof}

\begin{lemma}\label{lem:smallness}
Let $\lambda$ be an infinite cardinal.  If $Q^d$ is
$\operatorname{cf}(\lambda)$-chain-complete and $A\in Q$ has cardinality
$\lambda$, then $A$ has a subset $B\in Q$ with $|B|<\lambda$.
Consequently, if $Q^d$ is $\operatorname{cf}(\lambda)$-chain-complete for
every cardinal $\lambda$ with $\kappa\leq\lambda\leq|\D|$, then $Q$ is
$\kappa$-small.
\end{lemma}

\begin{proof}
Suppose that $A\in Q$, $|A|=\lambda$, and no smaller subset of $A$ belongs
to $Q$.  Enumerate $A=\set{a_\alpha:\alpha<\lambda}$ and put
$A_\alpha=\set{a_\beta:\beta<\alpha}$.  Then $A_\alpha\notin Q$, so
$A_\alpha^c\in Q^d$.  A cofinal subsequence of these complements is a
descending $\operatorname{cf}(\lambda)$-chain whose intersection is $A^c$.
Chain completeness would give $A^c\in Q^d$, contrary to $A\in Q$.

For the last assertion, repeatedly replace a member of $Q$ of cardinality at
least $\kappa$ by a strictly smaller member.  This process terminates because
the cardinals are well ordered.
\end{proof}

On a countable domain, the case $\lambda=\omega$, applied to $Q^d$, is the
implication from the descending chain condition (DCC) to the finite-witness
property (FIN) in Fact~4 of Altman, Peterzil, and Winter
\cite[Fact~4]{AltmanPeterzilWinter2005}.  They also prove the converse.  The
cardinal formulation above isolates the direction that will be used here.

\begin{lemma}\label{lem:small-complete}
Suppose that $P$ is $\kappa$-small and $Q$ is $\kappa$-complete.
Then $P\cdot Q=P\branch Q$.  In particular, $P\dom Q$.
\end{lemma}

\begin{proof}
If $R\in P\cdot Q$, let $A=\set{x:R_x\in Q}\in P$.  Choose
$B\in P$ with $B\subseteq A$ and $|B|<\kappa$.  Then
$C=\bigcap_{x\in B}R_x\in Q$ and $B\times C\subseteq R$.
\end{proof}

\begin{proposition}\label{prop:minimal}
Suppose $A$ is an inclusion-minimal member of $P$.  If $P\dom Q$, then $Q$ is
closed under intersections of $|A|$ members.  In particular, for
$\forall_A=\set{B\subseteq\D:A\subseteq B}$ we have
$\forall_A\dom Q$ if and only if $Q$ is closed under intersections of
$|A|$ members.
\end{proposition}

\begin{proof}
Write $A=\set{a_\alpha:\alpha<|A|}$ and let $X_\alpha\in Q$.  Give
$a_\alpha$ the row $X_\alpha$ and every point outside $A$ the empty row.
The resulting relation belongs to $P\cdot Q$.  A column belongs to $P$
exactly when it contains all of $A$, by minimality.  Dominance therefore says
$\bigcap_{\alpha<|A|}X_\alpha\in Q$.  The converse for $\forall_A$ follows
from Proposition~\ref{prop:branching}.
\end{proof}

The dual of $\forall_A$ is
$\exists_A=\set{B\subseteq\D:A\cap B\ne\varnothing}$.  If the inner
quantifier is principal, i.e., of the form $\forall_A$ for nonempty $A$, or if the outer quantifier is of the form
$\exists_A$, the lower branching-product equality holds without any
completeness assumption.  Together with duality and
Proposition~\ref{prop:minimal}, this
shows that Question~\ref{con:main} has a positive answer whenever one of
$P,P^d,Q,Q^d$ is principal.

Cardinality quantifiers give a further positive case, with an essential
restriction on scope order. On an infinite domain $\D$, and for a cardinal
$1\leq\kappa\leq|\D|$, put
\[
 Q^{\geq\kappa}=\{A\subseteq\D:|A|\geq\kappa\}.
\]
Its dual is
\[
 (Q^{\geq\kappa})^d
 =\{A\subseteq\D:|\D\setminus A|<\kappa\}.
\]

\begin{proposition}
\label{prop:inner-cardinality}
Let $\D$ be infinite and $1\leq\kappa\leq|\D|$. For every proper upward
monotone quantifier $P$ on $\D$,
\[
 P\dom Q^{\geq\kappa}
 \iff
 P=\exists_S\text{ for some nonempty }S\subseteq\D.
\]
When these conditions hold,
\[
 P\cdot Q^{\geq\kappa}=P\branch Q^{\geq\kappa}.
\]
\end{proposition}

\begin{proof}
Suppose $P\dom Q^{\geq\kappa}$. Since $\D$ is infinite, fix a partition
$(B_x)_{x\in\D}$ of $\D$ with $|B_x|=|\D|$ for every $x\in\D$.
For $A\in P$, define
\[
 R=\bigcup_{x\in A}(\{x\}\times B_x).
\]
Then $R_x=B_x$ for $x\in A$ and $R_x=\varnothing$ otherwise. Hence
$\{x\in\D:R_x\in Q^{\geq\kappa}\}=A$, so
$R\in P\cdot Q^{\geq\kappa}$. Dominance gives
\[
 \{y\in\D:R^y\in P\}\in Q^{\geq\kappa}.
\]
Since $\kappa\geq1$, some column belongs to $P$. Every column is either
empty or a singleton $\{x\}$ with $x\in A$. Properness of $P$ therefore
gives
\[
 (\forall A\in P)(\exists x\in A)\ \{x\}\in P.
\]
Put $S=\{x\in\D:\{x\}\in P\}$. The preceding observation and upward
monotonicity give $P=\exists_S$, and $S\ne\varnothing$ because $\D\in P$.

Conversely, suppose $P=\exists_S$ for some nonempty $S\subseteq\D$,
and let $R\in P\cdot Q^{\geq\kappa}$. Choose $x\in S$ such that
$R_x\in Q^{\geq\kappa}$. Since $\{x\}\in P$, the rectangle
$\{x\}\times R_x\subseteq R$ witnesses $R\in P\branch Q^{\geq\kappa}$.
Thus
\[
 P\cdot Q^{\geq\kappa}\subseteq P\branch Q^{\geq\kappa}.
\]
The reverse inclusion always holds, so the lower branching equality
follows, and hence dominance.
\end{proof}

\begin{example}\label{ex:cardinality}
On an uncountable domain, $Q_1=Q^{\geq\aleph_1}$ means ``uncountably many''.
Proposition~\ref{prop:inner-cardinality} shows that
$P\dom Q_1$ forces $P=\exists_S$ for some nonempty
$S\subseteq\D$, and hence \eqref{eq:lower}. This gives a positive
instance of Question~\ref{con:main} on every uncountable domain.

There is also a dual positive instance. On any infinite domain $\D$,
for $1\leq\kappa\leq|\D|$ and any proper upward monotone quantifier $P$,
Proposition~\ref{prop:duality} and the preceding proposition give
\[
 (Q^{\geq\kappa})^d\dom P
 \iff
 P^d\dom Q^{\geq\kappa}
 \iff
 P=\forall_S\text{ for some nonempty }S\subseteq\D.
\]
Here too \eqref{eq:lower} holds, since the inner quantifier is principal.

The scope order is essential: Remark~\ref{rem:outer-cardinality}
shows that $Q^{\geq\aleph_1}$ in the outer position can dominate
another quantifier while both branching equalities fail.
\end{example}

\section{What dominance forces}\label{sec:dominance}

The main necessary condition is a simultaneous obstruction to failures of
intersection closure.  Its proof works at every cofinality.

\begin{theorem}\label{thm:alternatives}
Assume $P\dom Q$.  Then:
\begin{enumerate}
\item either $P^d$ or $Q$ is a filter;
\item for every infinite regular cardinal $\kappa$, either $P^d$ or $Q$
      is $\kappa$-chain-complete.
\end{enumerate}
\end{theorem}

\begin{proof}
For (1), suppose $P^d$ is not closed under finite intersections. There are
disjoint $A,B\notin P$ such that $A\cup B\in P$. Indeed, starting with
$A_0,B_0\notin P$ and $A_0\cup B_0\in P$, replace them by
$A_0\setminus B_0$ and $B_0$; upward monotonicity shows that both remain
outside $P$. Given $C,E\in Q$, put
\[
 R=(A\times C)\cup(B\times E).
\]
Then $R\in P\cdot Q$. For $y\in\D$, the column $R^y$ belongs to $P$
exactly when $y\in C\cap E$: in that case $R^y=A\cup B\in P$, while
otherwise it is one of $A,B,\varnothing$, all outside $P$. Dominance
therefore gives $C\cap E\in Q$. Thus $Q$ is a filter.

For (2), suppose $(X_\alpha)_{\alpha<\kappa}$ is a descending chain in
$P^d$ whose intersection is not in $P^d$, and let
$(Y_\alpha)_{\alpha<\kappa}$ be a descending chain in $Q$. Put
\[
 R=\bigcup_{\alpha<\kappa}(X_\alpha^c\times Y_\alpha).
\]
The set of points with a $Q$-large row is
\[
 \bigcup_{\alpha<\kappa}X_\alpha^c
 =\left(\bigcap_{\alpha<\kappa}X_\alpha\right)^c\in P.
\]
Indeed, if $x$ belongs to this union and $\delta$ is least such that
$x\notin X_\delta$, then $R_x=Y_\delta\in Q$; otherwise $R_x=\varnothing$.

For fixed $y$, the column $R^y$ belongs to $P$ exactly when
$y\in\bigcap_{\alpha<\kappa}Y_\alpha$. In that case
\[
 R^y=\bigcup_{\alpha<\kappa}X_\alpha^c
     =\left(\bigcap_{\alpha<\kappa}X_\alpha\right)^c\in P.
\]
Otherwise, if $\delta$ is least such that $y\notin Y_\delta$, then
\[
 R^y=\bigcup_{\alpha<\delta}X_\alpha^c
 \subseteq X_\delta^c\notin P.
\]
Dominance therefore gives
$\bigcap_{\alpha<\kappa}Y_\alpha\in Q$.
\end{proof}

For $\kappa=\omega$, the two conclusions are the finite-intersection and DCC
conditions in the necessity direction of
\cite[Theorem~5]{AltmanPeterzilWinter2005}.  The proof above separates their
two relation constructions and shows that the chain argument is not
intrinsically countable.

The alternatives in Theorem~\ref{thm:alternatives} are independent
disjunctions.  In particular, the theorem does \emph{not} say that one of
$P^d,Q$ is chain-complete at every regular cardinal.  This distinction is
the central issue on uncountable domains.

\begin{corollary}\label{cor:uniform}
Let $\kappa$ be an infinite cardinal.
\begin{enumerate}
\item If $P$ is $\kappa$-small and $Q$ is $\kappa$-complete, then
      \eqref{eq:lower} holds.
\item If $Q^d$ is $\kappa$-small and $P^d$ is $\kappa$-complete, then
      \eqref{eq:upper} holds.
\item If $Q$ is $\operatorname{cf}(\lambda)$-chain-complete for every
      cardinal $\lambda$ with $\kappa\leq\lambda\leq|\D|$, and $P^d$
      is $\kappa$-complete, then \eqref{eq:upper} holds.
\item Dually, if $P^d$ is $\operatorname{cf}(\lambda)$-chain-complete
      for every cardinal $\lambda$ with
      $\kappa\leq\lambda\leq|\D|$, and $Q$ is $\kappa$-complete, then
      \eqref{eq:lower} holds.
\end{enumerate}
\end{corollary}

\begin{proof}
The first two clauses are Lemma~\ref{lem:small-complete}, with the second
applied to $Q^d,P^d$. For (3), Lemma~\ref{lem:smallness} makes $Q^d$
$\kappa$-small, so (2) applies. Clause (4) is dual.
\end{proof}

\section{The countable characterization}\label{sec:countable}

On a countable domain, the preceding alternatives reduce to their
$\omega$-versions, and the argument closes.

\begin{theorem}\label{thm:countable}
Let $\D$ be countable and let $P,Q$ be proper upward monotone
quantifiers on $\D$.  Then
\[
 P\dom Q \iff  \text{\eqref{eq:lower} or \eqref{eq:upper} holds.}
\]
\end{theorem}

\begin{proof}
The reverse implication was observed after Question~\ref{con:main}.
Suppose $P\dom Q$ and put $F=P^d$ and $G=Q$.  By
Theorem~\ref{thm:alternatives},  at least one of $F,G$ is a filter,
and at least one of $F,G$ is $\omega$-chain-complete.

Combining these two alternatives gives four cases. If $F$ is both a filter and
$\omega$-chain-complete, it is closed under countable intersections.  Since
$\D$ is countable,
\[
 A=\bigcap\set{\D\setminus\{a\}:\D\setminus\{a\}\in F}
\]
belongs to $F$.  Moreover, $A\subseteq B$ for every $B\in F$: if
$a\notin B$, then $B\subseteq\D\setminus\{a\}$, so
$\D\setminus\{a\}\in F$. Hence
\[
 F=\set{B\subseteq\D:A\subseteq B}.
\]
Thus $F=\forall_A$ and $P=F^d=\exists_A$, so \eqref{eq:lower} holds.
If $G$ is both a filter and
$\omega$-chain-complete, the same argument makes $G$ principal, and again
\eqref{eq:lower} holds.

Suppose next that $F$ is a filter and $G$ is
$\omega$-chain-complete.  By Lemma~\ref{lem:smallness}, $G^d=Q^d$ is
$\omega$-small.  Every filter is $\omega$-complete in the standard
``fewer than $\omega$'' sense.  Lemma~\ref{lem:small-complete}, applied to
$Q^d$ and $F$, gives \eqref{eq:upper}.  Finally, if $G$ is a filter and
$F$ is $\omega$-chain-complete, Lemma~\ref{lem:smallness} makes
$F^d=P$ $\omega$-small, and Lemma~\ref{lem:small-complete} gives
\eqref{eq:lower}.
\end{proof}

\begin{remark}
Theorem~\ref{thm:countable} gives the characterization stated as
Theorem~5 of Altman, Peterzil, and Winter
\cite{AltmanPeterzilWinter2005}. 
Their DCC is
our $\omega$-chain-completeness, their FIN is our $\omega$-smallness, and
their nontriviality clause is automatic under our standing properness
assumption.  
Thus the proof above follows the same four-case division.  The
branching-product formulation makes the conclusion of each mixed case
explicit: dominance is witnessed either by a $P$--$Q$ rectangle or,
after duality, by a $Q^d$--$P^d$ rectangle.

The proof above repairs the gap by a different route.  DCC is converted,
via Lemma~\ref{lem:smallness}, into the existence of finite witnesses,
and Lemma~\ref{lem:small-complete} then produces the required rectangle;
duality handles the other mixed case.  Thus the argument gives a corrected
proof of the countable characterization while isolating two ingredients
that extend to arbitrary cardinalities.
\end{remark}

\section{A counterexample}\label{sec:counterexample}

The positive cases above do not extend to all upward monotone quantifiers.
The reason is elementary: dominance is preserved by unions of outer
quantifiers, even when different components use different branching
alternatives.  The following observation makes this precise.

\begin{lemma}\label{lem:union-dominance}
Let $(P_i)_{i\in I}$ be a nonempty family of proper upward monotone
quantifiers on $\D$, and let $Q$ be such a quantifier.  If $P_i\dom Q$
for every $i\in I$, then
\[
 \bigcup_{i\in I}P_i\dom Q.
\]
\end{lemma}

\begin{proof}
If $A=\{x:R_x\in Q\}\in\bigcup_{i\in I}P_i$, choose $i$ with $A\in P_i$.
Then $\{y:R^y\in P_i\}\in Q$.  This set is contained in
$\{y:R^y\in\bigcup_{i\in I}P_i\}$, which therefore belongs to $Q$ by
upward monotonicity.
\end{proof}

\begin{theorem}\label{thm:counterexample}
There are proper upward monotone quantifiers $P,Q$ on a domain of
cardinality $\aleph_2$ such that $Q$ is a filter and $P\dom Q$,
but neither \eqref{eq:lower} nor \eqref{eq:upper} holds.
\end{theorem}

\begin{proof}
Let $\D=X\sqcup Y$, where $|X|=\aleph_0$ and $|Y|=\aleph_2$, and define
\begin{align*}
 P&=\{A\subseteq\D:X\setminus A\text{ is finite}
                \text{ or }|A\cap Y|=\aleph_2\},\\
 Q&=\{A\subseteq\D:\D\setminus A\text{ is countable}\}.
\end{align*}
Both quantifiers are proper and upward monotone, and $Q$ is the
co-countable filter.  Their duals are
\begin{align*}
 P^d&=\{A\subseteq\D:A\cap X\text{ is infinite}
                    \text{ and }|Y\setminus A|\leq\aleph_1\},\\
 Q^d&=\{A\subseteq\D:A\text{ is uncountable}\}.
\end{align*}

\smallskip
\noindent\emph{Dominance.}
Write $P=P_0\cup P_1$, where
\[
 P_0=\{A\subseteq\D:X\setminus A\text{ is finite}\},\qquad
 P_1=\{A\subseteq\D:|A\cap Y|=\aleph_2\}.
\]
The quantifier $P_0$ is $\aleph_1$-small: for $A\in P_0$, the countable
set $A\cap X$ is still in $P_0$.  Since $Q$ is $\aleph_1$-complete,
Lemma~\ref{lem:small-complete} gives
\[
 P_0\cdot Q=P_0\branch Q,
\]
so $P_0\dom Q$.

On the other hand, $Q^d$ is $\aleph_2$-small, and
\[
 P_1^d=\{A\subseteq\D:|Y\setminus A|\leq\aleph_1\}
\]
is $\aleph_2$-complete: a union of at most $\aleph_1$ sets, each of
cardinality at most $\aleph_1$, still has cardinality at most $\aleph_1$.
The same lemma therefore gives
\[
 Q^d\cdot P_1^d=Q^d\branch P_1^d.
\]
By duality, $P_1\dom Q$.  Lemma~\ref{lem:union-dominance} now yields
$P\dom Q$.  Thus the first component uses the lower branching
alternative, whereas the second uses the upper one.

\smallskip
\noindent\emph{Failure of \eqref{eq:lower}.}
Consider the relation
\[
 T=\{(x,y)\in\D^2:x\in Y\text{ and }x\ne y\}.
\]
Its rows indexed by $Y$ are co-countable and its other rows are empty.
Since $Y\in P$, we have $T\in P\cdot Q$.
Suppose $C\times B\subseteq T$, with $C\in P$ and $B\in Q$.
Since $B\ne\varnothing$, the empty rows force $C\subseteq Y$.
Membership in $P$ then gives $|C|=\aleph_2$.  Since $B$ is co-countable,
there is $z\in C\cap B$.  But $(z,z)\in C\times B$ and $(z,z)\notin T$,
a contradiction.  Hence $T\notin P\branch Q$.

\smallskip
\noindent\emph{Failure of \eqref{eq:upper}.}
Choose infinite subsets $(H_\alpha)_{\alpha<\omega_1}$ of $X$ whose
pairwise intersections are finite.  Such a family can be constructed
explicitly in ZFC: identify $X$ with $2^{<\omega}$, the set of finite
binary sequences, choose distinct branches $r_\alpha\in2^\omega$ for
$\alpha<\omega_1$, and put
\[
 H_\alpha=\{r_\alpha\restriction n:n<\omega\}.
\]
There are at least $\aleph_1$ branches because $2^\omega$ is uncountable,
and two distinct branches have only finitely many common initial segments.

Choose distinct $k_\alpha\in Y$ for $\alpha<\omega_1$, and define a
relation $S$ by
\[
 S_{k_\alpha}=Y\cup H_\alpha\quad(\alpha<\omega_1),\qquad
 S_x=\varnothing\quad(x\notin\{k_\alpha:\alpha<\omega_1\}).
\]
Each $S_{k_\alpha}$ belongs to $P^d$.  There are uncountably many such
rows, so $S\in Q^d\cdot P^d$.
Suppose $C\times B\subseteq S$, with $C\in Q^d$ and $B\in P^d$.
Since $B\cap X$ is infinite, $B$ is nonempty, and the empty rows force
$C\subseteq\{k_\alpha:\alpha<\omega_1\}$.  The uncountable set $C$
contains distinct $k_\alpha,k_\beta$.  
Since $k_\alpha,k_\beta\in C$ and $C\times B\subseteq S$, we have
\[
 B\subseteq S_{k_\alpha}\cap S_{k_\beta}
 =(Y\cup H_\alpha)\cap(Y\cup H_\beta)
 =Y\cup(H_\alpha\cap H_\beta).
\]
Intersecting with $X$ therefore gives
\[
 B\cap X\subseteq H_\alpha\cap H_\beta,
\]
which is finite, contradicting $B\in P^d$.
 Thus
$S\notin Q^d\branch P^d$, and \eqref{eq:upper} fails.
\end{proof}

The counterexample uses only one filter: $Q$ is a filter, whereas $P$ is not,
since two disjoint subsets of $Y$ of cardinality $\aleph_2$ both belong to
$P$.  The case in which both quantifiers are filters is therefore separate
and will be treated in the next section.  

\begin{remark}
\label{rem:outer-cardinality}
The reversed pair of duals gives a counterexample with an ordinary
cardinality quantifier in the outer position. Write
$Q_1=Q^{\geq\aleph_1}$ for the quantifier ``uncountably many''.
In the construction above, $Q^d=Q_1$, so
Proposition~\ref{prop:duality} gives
\[
 Q_1\dom P^d,\qquad
 P^d=\{A\subseteq\D:A\cap X\text{ is infinite}
                   \text{ and }|Y\setminus A|\leq\aleph_1\}.
\]
For this pair the lower equality is the original upper equality, and the
upper equality is the original lower equality; both therefore fail.
Consequently, Example~\ref{ex:cardinality} cannot be extended merely by
requiring that one of the two quantifiers be a cardinality quantifier.
\end{remark}

\section{The two-filter case}\label{sec:filters}

The elementary counterexample of Section~\ref{sec:counterexample} has only
one filter.  The restriction of Question~\ref{con:main} to two filters has a
different status: it holds under $V=L$, while a measurable cardinal yields
a counterexample.  In fact, the existence of a two-filter counterexample
has exactly the consistency strength of a measurable cardinal.

The positive direction uses Donder's regularity theorem
\cite[Theorems~4.1 and~4.5]{Donder1988}; the formulation needed here is also
stated in \cite[Theorem~8.1]{Usuba2026}.  An ultrafilter on a set of
cardinality $\mu$ is \emph{uniform} if all its members have cardinality
$\mu$.  It is $(\omega,\lambda)$-regular if it contains a family
$(B_\xi)_{\xi<\lambda}$ such that each point belongs to only finitely many
$B_\xi$.  Donder's theorem says that, if there is no inner model with a
measurable cardinal, every uniform ultrafilter on an uncountable $\mu$ is
$(\omega,\mu)$-regular when $\mu$ is singular, and is
$(\omega,\lambda)$-regular for every infinite $\lambda<\mu$ when $\mu$ is
regular.  Consequently, in either case it supplies such a point-finite
family of length $\lambda$ whenever $\lambda$ is infinite and $\lambda<\mu$.
We use this theorem as the set-theoretic input.

\begin{lemma}\label{lem:filter-positive-chains}
Assume there is no inner model with a measurable cardinal.  Let $F$ be a
filter on $\D$, and suppose
\[
 \mu=\min\{|A|:A\in F\}
\]
is infinite.  For every infinite regular $\lambda\leq\mu$, there is a
descending $\lambda$-chain of members of $F^d$ with empty intersection.
\end{lemma}

\begin{proof}
Choose $A\in F$ of cardinality $\mu$.  Every member of the  filter
\[
 F\restriction A=\{C\cap A:C\in F\}
\]
has cardinality $\mu$, by the minimality of $\mu$.  Now, the set 
\[
 (F\restriction A) \cup \set{ A \setminus S  : S \subseteq A, |S| < \mu }
 \]
has the finite intersection property: the intersection of finitely many trace members
still has size $\mu$, and removing fewer than $\mu$ points cannot make it
empty.  Extend this set to an ultrafilter $W$ on $A$.  It is
uniform, since it contains the complement of every subset of $A$ of size
less than $\mu$.

Every $B\in W$, regarded as a subset of $\D$, belongs to $F^d$.
Indeed, if $\D\setminus B\in F$, then
\[
 A\setminus B=A\cap(\D\setminus B)\in F\restriction A\subseteq W,
\]
contradicting $B\in W$, since $W$ is a proper filter on $A$.

If $\lambda=\mu$ (so $\mu$ is regular), enumerate
$A=\{a_\xi:\xi<\mu\}$ and use the tails
\[
 T_\alpha=A\setminus\{a_\xi:\xi<\alpha\}\qquad(\alpha<\mu).
\]
Uniformity gives $T_\alpha\in W$, and their intersection is empty.  This
also deals with the case $\mu=\aleph_0$.

If $\lambda<\mu$, then $\mu$ is uncountable, and Donder's theorem supplies
a point-finite family $(B_\xi)_{\xi<\lambda}$ in $W$.  Put
\[
 T_\alpha=\bigcup_{\alpha\leq\xi<\lambda}B_\xi
 \qquad(\alpha<\lambda).
\]
Each $T_\alpha$ contains $B_\alpha$ and hence belongs to $W$.  The chain
is descending, and its intersection is empty because a point belongs to
only finitely many $B_\xi$.  In both cases it is a chain in $F^d$ as
required.
\end{proof}

\begin{theorem}
\label{thm:filters-no-inner-measurable}
Assume there is no inner model with a measurable cardinal.  Let $F,G$ be
filters on $\D$.  Then
\[
 F\dom G\iff F\cdot G=F\branch G 
\]
\end{theorem}

\begin{proof}
Suppose first that $\mu=\min\{|A|:A\in F\}$ is infinite and $F\dom G$.  By
Lemma~\ref{lem:filter-positive-chains}, $F^d$ fails
$\lambda$-chain-completeness for every infinite regular
$\lambda\leq\mu$.  The chain alternative in
Theorem~\ref{thm:alternatives} therefore makes $G$
$\lambda$-chain-complete for every such $\lambda$.  Since $G$ is a filter,
Lemma~\ref{lem:chain-to-complete}, with $\kappa=\mu^+$, shows that $G$ is
$\mu^+$-complete.

Choose $A\in F$ with $|A|=\mu$.  For every $C\in F$, the set
$C\cap A$ still belongs to $F$ and has cardinality at most $\mu$.
Thus $F$ is $\mu^+$-small.  Since $G$ is $\mu^+$-complete,
Lemma~\ref{lem:small-complete} gives
\[
 F\cdot G=F\branch G.
\]

If $\mu$ is finite, choose $A\in F$ with $|A|=\mu$.  For every
$C\in F$, we have $A\cap C\in F$; by the minimality and finiteness of
$|A|$, this forces $A\cap C=A$.  Hence $A\subseteq C$ for every
$C\in F$, and therefore
\[
 F=\forall_A.
\]
Since $G$ is closed under finite intersections,
Proposition~\ref{prop:branching} gives
$F\cdot G=F\branch G$.
\end{proof}

Note that the proof gives a stronger statement: 
\[
 F\dom G\iff F\cdot G=F\branch G \iff G \text{ is $\mu^+$-complete,}
\]
when $\mu=\min\{|A|:A\in F\}$ is infinite.

Under $V=L$, every inner model is $L$, since $L$ is contained in every
inner model.  Scott's theorem \cite{Scott1961} shows that $L$ has no
measurable cardinal.  Hence the hypothesis of the theorem holds under
$V=L$.

The next construction shows why the inner-model hypothesis cannot simply
be dropped.  It uses a finite intersection of outer filters rather than
the union of Section~\ref{sec:counterexample}.  Finite intersections of
outer quantifiers preserve dominance when the inner quantifier is a
filter, so the two branching mechanisms can still be combined without
losing finite-intersection closure.

\begin{theorem}
\label{thm:measurable-filter-counterexample}
If there is a measurable cardinal $\kappa$, then there are two proper
non-maximal filters $F,G$ on a domain of cardinality $\kappa$ such that
$F\dom G$, but
\[
 F\cdot G\ne F\branch G,
 \qquad
 G^d\cdot F^d\ne G^d\branch F^d.
\]
\end{theorem}

\begin{proof}
Let $\kappa>\aleph_1$ be measurable, and let $U$ be a nonprincipal
$\kappa$-complete ultrafilter on a set $Y$ of cardinality $\kappa$.
Such an ultrafilter is uniform: if $B\subseteq Y$ has size less than
$\kappa$, intersecting the complements of its singletons shows
$Y\setminus B\in U$.
Let $\D=X\sqcup Y$, where $X$ is countably infinite, and put
\begin{align*}
 F&=\{A\subseteq\D:X\setminus A\text{ is finite and }A\cap Y\in U\},\\
 G&=\{A\subseteq\D:\D\setminus A\text{ is countable}\}.
\end{align*}
Both families are proper filters.  The filter $F$ is not an ultrafilter:
if $X=X_0\sqcup X_1$ with both parts infinite, neither $Y\cup X_0$ nor
its complement belongs to $F$.  The co-countable filter $G$ is also
non-maximal, as a partition of $\D$ into two uncountable sets shows.
Their duals are
\begin{align*}
 F^d&=\{A\subseteq\D:A\cap X\text{ is infinite or }A\cap Y\in U\},\\
 G^d&=\{A\subseteq\D:A\text{ is uncountable}\}.
\end{align*}

\smallskip
\noindent\emph{Dominance.}
Write $F=F_0\cap\widehat U$, where
\[
 F_0=\{A\subseteq\D:X\setminus A\text{ is finite}\},
 \qquad
 \widehat U=\{A\subseteq\D:A\cap Y\in U\}.
\]
As in Section~\ref{sec:counterexample}, $F_0$ is $\aleph_1$-small and
$G$ is $\aleph_1$-complete.  Hence $F_0\dom G$ by
Lemma~\ref{lem:small-complete}.  Also $G^d$ is $\aleph_2$-small, whereas
$\widehat U^d=\widehat U$ is $\aleph_2$-complete.  The same lemma and
duality give $\widehat U\dom G$.

To combine these two implications, suppose $R\in F\cdot G$.
Its set of $G$-large rows belongs to both $F_0$ and $\widehat U$.
Therefore both
\[
 J_0=\{y:R^y\in F_0\},\qquad
 J_1=\{y:R^y\in\widehat U\}
\]
belong to $G$.  Since $G$ is a filter,
$J_0\cap J_1=\{y:R^y\in F\}\in G$.  Thus $F\dom G$.

\smallskip
\noindent\emph{Failure of the lower equality.}
Let $T=\{(x,y)\in\D^2:x\ne y\}$.  Every row is co-countable, so
$T\in F\cdot G$.  If $C\in F$ and $B\in G$, then
$C\cap Y\in U$ has cardinality $\kappa$, and hence $C\cap B$ is
nonempty.  Any $z\in C\cap B$ gives a diagonal point in $C\times B$
which is not in $T$.  Thus $T$ contains no $F$--$G$ rectangle.

\smallskip
\noindent\emph{Failure of the upper equality.}
Choose an almost disjoint family $(H_\alpha)_{\alpha<\omega_1}$ of
infinite subsets of $X$, as constructed in the proof of
Theorem~\ref{thm:counterexample}, and distinct points
$k_\alpha\in Y$ for $\alpha<\omega_1$.  Define a relation $S$ by
\[
 S_{k_\alpha}=H_\alpha\quad(\alpha<\omega_1),\qquad
 S_x=\varnothing\quad(x\notin\{k_\alpha:\alpha<\omega_1\}).
\]
Every $H_\alpha$ belongs to $F^d$, so $S\in G^d\cdot F^d$.
Suppose $C\times B\subseteq S$, with $C\in G^d$ and $B\in F^d$.
The empty rows and $B\ne\varnothing$ force
$C\subseteq\{k_\alpha:\alpha<\omega_1\}$.  Since $C$ is uncountable,
it contains distinct $k_\alpha,k_\beta$, whence
\[
 B\subseteq H_\alpha\cap H_\beta.
\]
This makes $B$ a finite subset of $X$, so neither clause in the definition
of $F^d$ holds.  The contradiction proves that $S$ contains no
$G^d$--$F^d$ rectangle.
\end{proof}

\begin{corollary}\label{cor:filter-consistency}
The existence of a counterexample to Question~\ref{con:main} among pairs
of filters is equiconsistent with the existence of a measurable cardinal.
Moreover, under $V=L$ no such counterexample exists.
\end{corollary}

\begin{proof}
By Theorem~\ref{thm:filters-no-inner-measurable}, any such counterexample
yields an inner model with a measurable cardinal.  Conversely,
Theorem~\ref{thm:measurable-filter-counterexample} produces such a
counterexample from a measurable cardinal.  The $V=L$ case follows from
Theorem~\ref{thm:filters-no-inner-measurable} and Scott's theorem.
\end{proof}

\section{Concluding remarks} 

The branching product provides a common, transpose-symmetric core for the two
iterated scope readings.  This separates the problem into three parts:
dominance between the iterations, reduction to branching, and the
completeness or smallness conditions that force such a reduction.  On
countable domains, dominance always yields one of the two branching
equalities.  On arbitrary domains, several important positive cases remain:
principal cases and the smallness/completeness conditions of Corollary~\ref{cor:uniform}, the inner-cardinality case of
Proposition~\ref{prop:inner-cardinality}, and, under GCH, the ultrafilter
case by Goldberg's theorem.

For general upward monotone quantifiers, however, Theorem~\ref
{thm:counterexample} shows that the dichotomy fails in ZFC.  The
counterexample exploits the fact that unions of outer quantifiers preserve
dominance, allowing the lower and upper branching mechanisms to be combined
while neither survives globally. By Remark~\ref{rem:outer-cardinality}, the
failure can even occur with ``uncountably many'' in outer scope.

For pairs of filters, the situation is more rigid.  In the absence of an inner
model with a measurable cardinal, Theorem~\ref
{thm:filters-no-inner-measurable} shows that dominance already implies the
lower branching equality.  
In particular, $V=L$ gives a positive
answer to Question~\ref{con:main} for all pairs of filters.

At the opposite end, a measurable cardinal yields the two-filter
counterexample of Theorem~\ref{thm:measurable-filter-counterexample}.  
Together with Donder's theorem, this gives the exact consistency strength
stated in Corollary~\ref{cor:filter-consistency}.  Thus the general problem
is already settled negatively in ZFC, while the two-filter case is
independent of ZFC, relative to the consistency of a measurable cardinal.
What remains is a finer structural analysis of dominance on arbitrary
domains, especially in universes admitting inner models with measurable
cardinals.

\section*{Acknowledgements and AI declaration}
\label{sec:ai-declaration}

The author used \mbox{ChatGPT} (OpenAI) during the development and revision
of this paper for mathematical exploration, literature searches, and
text editing.  The tool was prompted
with drafts of the manuscript and questions about proofs, counterexamples,
and special cases of the branching-product dichotomy.  Its assistance included proposing and developing the
inner-cardinality characterization in Proposition~\ref{prop:inner-cardinality},
the counterexample in Section~\ref{sec:counterexample},
and the arguments in Section~\ref{sec:filters} involving Donder's theorem,
a measurable-cardinal counterexample, and consistency strength.  
The author retains sole responsibility for the content of the paper,
including verification of all mathematical statements and proofs.

\section*{Funding}

The author was supported by the Swedish Research Council
(Vetenskapsrådet) under grant 2022-01685.

\bibliographystyle{abbrv}
\bibliography{refs}

\end{document}